\documentclass{article}

\usepackage{amsthm}
\usepackage{amsmath}

\newtheorem{thm}{Theorem}
\newtheorem{prop}{Proposition}
\newtheorem{lem}{Lemma}
\newtheorem{cor}{Corollary}

\newcommand{\Real}{{\bf R}}
\newcommand{\Comp}{{\bf C}}

\newcommand{\Del}{\partial}
\newcommand{\Ann}{{\rm Ann}}
\newcommand{\LT}{{\rm LT}}

\makeatletter
  
  \@addtoreset{equation}{section}
\makeatother

\title{
  A Holonomic Ideal Annihilating the Fisher--Bingham Integral
  \footnote{MSC classes: 16S32, 16Z99,32C38,62F10}
}
\author{Tamio Koyama}
\date{}

\begin{document}
\maketitle
\begin{abstract}
We calculate the integration ideal of annihilating differential operators
of the non-normalized Fisher--Bingham distribution
and show that the ideal agrees with the set of operators for 
the Fisher--Bingham integral given in \cite{Sei}.
They conjectured that the set generates a holonomic ideal
and we prove their conjecture.
\end{abstract}

\section{Introduction}\label{intro}
The Fisher--Bingham distribution is a probability distribution on
the $n$-dimensional sphere $S^n(r)$ with the radius  $r$ defined by  
\begin{equation}\label{FBdist}
\frac{1}{F(x, y, r)} \exp(t^Txt+yt)|dt|.
\end{equation}
Here, the variable $x$ is an $(n+1)\times (n+1)$ symmetric matrix whose $(i, j)$ component is
$x_{ij}$ when $i=j$ and $x_{ij}/2$ when $i\neq j$.
The variable $y$ (resp. $t$) is a row (resp. column) vector  of length $n+1$, 
and the measure $|dt|$ is the Haar measure on $S^n(r)$.
The function $F(x, y, r)$ is the normalizing  constant defined by
\begin{equation}\label{FB}
F(x, y, r) = \int_{S^n(r)} \exp
\left(
\sum_{1\leq i \leq j \leq n+1}x_{ij}t_it_j 
+\sum_{i=1}^{n+1}y_it_i
\right)
|dt|.
\end{equation}
The integral $(\ref{FB})$ is referred to as the Fisher--Bingham integral on the sphere
$S^n(r)$.

The Fisher--Bingham distribution is used in directional statistics.
Kent studied  estimations, hypothesis testings 
and confidence regions 
with respect to the Fisher--Bingham distribution on
the $2$-dimensional sphere \cite{Kent}, and
in the book by Mardia and Jupp on directional statistics \cite[chapter 9]{Mardia},
 a definition of the Fisher--Bingham distribution having 
the same form as $(\ref{FBdist})$ and a relation with an another 
probability distribution on the sphere are explained.

We are interested in estimating the value of parameters $x$ and $y$ which maximizes
the likelihood function
$$
f(x, y) = \frac{1}{F(x, y, r)^N}
\prod_{i=1}^N \exp(t_i^Txt_i+yt_i)
$$
for given data $t_1, \cdots, t_N\in S^n$.
This problem is equivalent to estimating the value of parameters $x$ and $y$
which minimizes the function
$$
F(x, y, r)\exp\left(-\sum_{1\leq i\leq j\leq n}S_{ij}x_{ij}-S_iy_i \right)
$$
for given
$
\{S_{ij}\}_{1\leq i \leq j \leq n}, \, 
\{S_i\}_{1\leq i \leq  n} 
\subset \Real.
$
There are several approaches to solving this problem.
Among them, 
the holonomic gradient descent proposed in \cite{Sei} enables us to
estimate the value by utilizing
linear partial differential operators with polynomial 
coefficients which annihilate the Fisher--Bingham integral $(\ref{FB})$
and generate a holonomic ideal.
Let $D_d$ be the ring of differential operators
$D_d = {\bf C}\langle z_1, \dots, z_d, \Del_1, \dots, \Del_d\rangle$.
A left ideal in $D_d$ is called a holonomic ideal when the characteristic ideal 
${\rm in}_{(0, 1)}(I)$ generated by the principal symbols of $I$ in 
${\bf C}[z_1, \dots, z_d, \xi_1, \dots, \xi_d]$ 
has the Krull dimension $d$.
See, e.g., \cite[p 31, Definition 1.4.8]{SST}, \cite{Oaku1}, 
and their references for details.

In \cite{Sei}, it is shown that 
the Fisher--Bingham integral $F(x, y, r)$ is annihilated by the following 
linear partial differential operators.
\begin{equation}\label{annFB}
\begin{split}
  &\Del _{x_{ij}} - \Del _{y_i}\Del_{y_j}\quad (i\leq j), \\ 
  &\sum _{i=1}^{n+1} \Del_{x_{ii}}- r^2, \\ 
  &x_{ij}\Del_{x_{ii}}+2(x_{jj}-x_{ii})\Del _{x_{ij}}
	-x_{ij}\Del_{x_{jj}} \\
  &\quad	+\sum _{k \neq i, j}
	\left( x_{kj}\Del_{x_{ik}}-x_{ik}\Del_{x_{jk}}\right)
	+y_j\Del_{y_i} -y_i\Del _{y_j}\quad (i < j, x_{k\ell}=x_{\ell k}), \\
  & r\Del_r -2\sum _{i\leq j}x_{ij}\Del_{x_{ij}}
	- \sum_i y_i\Del _{y_i} -n. 
\end{split}
\end{equation}
They also show that $(\ref{annFB})$ generates a holonomic ideal in the
cases $n=1$ and $n=2$ by a calculation on a computer, 
and conjecture that it holds for any $n$.
We will prove this conjecture.

In order to state the main result of this paper precisely, 
let us explain the notion of the integration ideal.
For a holonomic ideal $I$ in $D_d$,
the left ideal 
$(I+\Del_{d'+1}D_d+\cdots+\Del_dD_d)\cap D_{d'}$ 
in $D_{d'}$ is called the integration ideal and it is known that
the integration ideal is a holonomic ideal in $D_{d'}$
(see, e.g., \cite[Chapter 1]{Bjork}, \cite[\S 5.5]{SST}).

In the present paper, we show that $(\ref{annFB})$ generates the integration ideal of
the annihilating ideal
$${\rm Ann}\left(\exp\left(
\sum_{1\leq i \leq j \leq n+1}x_{ij}t_it_j 
+\sum_{i=1}^{n+1}y_it_i
\right)
|dt|\right).
$$
Here, 
$
\{z_1, \dots, z_{d'}\}
 = \{x_{ij}, y_k \vert 1\leq i \leq j \leq n+1, 1\leq k  \leq n+1\}
$
and 
$
\{z_{d'+1}, \dots, z_d\} = \{t_1 , \dots, t_{n+1}\}.
$
As its corollary, we show that $(\ref{annFB})$ generates a 
holonomic ideal for any $n$ and prove the conjecture in \cite{Sei}. 
Oaku gave an algorithm for computing the integration ideal in \cite{Oaku2}. 
The proof for $n=1, 2$ are done by applying this algorithm on a computer.
We apply this algorithm for a general natural number $n$, for which the steps of the algorithm cannot necessarily be applied, and so some propositions are necessary.

In section \ref{section2}, we consider the holonomic ideal annihilating 
the Haar measure on $S^n(r)$.
In section \ref{section3}, we give generators of the holonomic ideal
which annihilates the integrand of the Fisher--Bingham integral.
In section \ref{section4}, 
we compute the integration ideal of the holonomic ideal
which is given in section \ref{section3}, 
and prove the main theorem of this paper.

\section{The Haar measure on $S^n(r)$}\label{section2}
The Riemannian metric on the $n$-dimensional sphere with radius $r$ 
is constructed
by the pullback of the standard metric 
on the $(n+1)$-dimensional Euclidean space 
$\Real ^{n+1}$ along the embedding map. 
The metric defines 
a probability measure on $S^n(r)$, which is called the Haar measure and denoted by $|dt|$.
We define a distribution $\mu_r$ with a parameter $r>0$ as 
$$
\langle \mu_r, \varphi (t) \rangle 
:=
\int_{S^{n}(r)} \varphi |dt|.
$$
Here, $\varphi(t)$ is a test function.

Let $D={\bf C}\langle x, y, r, t, \Del_x, \Del_y, \Del_r, \Del_t\rangle$ 
be the 
ring of differential operators with polynomial coefficients.
For a given distribution $F$, we denote by $\Ann(F)$ the set of the operators in $D$
which annihilate $F$.
\begin{lem}
A left ideal $I$ in $D$ generated by following differential operators is
a subset of $\Ann(\mu_r)$.
\begin{equation}\label{annMU}
\begin{split}
  &      \Del_{x_{ij}}                    \, (1\leq i \leq j \leq n+1), 
  \quad  \Del_{y_i}                       \, (1\leq i \leq n+1), 
  \quad  t_1^2+\cdots+t_{n+1}^2-r^2, \\
  &      t_i\Del_{t_j}-t_j\Del_{t_i}  \, (1\leq i < j \leq n+1), 
       \quad r\Del_r+\sum_{i=1}^{n+1}t_i\Del_i+1
\end{split}
\end{equation}
\end{lem}
For computing the integration ideal, 
the following proposition is important. 
\begin{prop}
The left ideal $I$ in $D$ is a holonomic ideal.
\end{prop}
This proposition may be well known, however we could not find a proof
in the literature.
Therefore, we present a proof here. 
\begin{proof}
By the fundamental theorem of algebraic analysis
(see, e.g., \cite[p30.Theorem1.4.5]{SST}), 
it is enough to show that the dimension of the characteristic ideal 
${\rm in}_{(0, e)}(I)$ is not more than the number of variables
$N:=n(n+1)/2+2n+1$.

We can find the operators
$
r^2\Del_{t_k}+t_kr\Del_r-t_k  \, (1 \leq k \leq n+1)
$
in $I$ as follows.
\begin{eqnarray*}
&&
t_{n+1}(t_{n+1}\Del_{t_{n+1}}+\cdots+t_1\Del_{t_1}+r\Del_r+1)
-\Del_{t_{n+1}}(t_{n+1}^2+\cdots+t_1^2-r^2)\\
&=&
-\sum_{i=1}^n t_i(t_i\Del_{t_{n+1}}-t_{n+1}\Del_{t_i})
+r^2\Del_{t_{n+1}}+t_{n+1}r\Del_r-t_{n+1}, \\
&&
t_k(t_{n+1}\Del_{t_{n+1}}+\cdots+t_1\Del_{t_1}+r\Del_r+1)
-t_{n+1}(t_k\Del_{t_{n+1}}-t_{n+1}\Del_{t_k})\\
&=&
-\sum_{i=1}^{k-1} t_i(t_i\Del_{t_k}-t_k\Del_{t_i})
+\sum_{i=k+1}^n t_i(t_k\Del_{t_i}-t_i\Del_{t_k})
+\Del_{t_k}(t_{n+1}^2+\cdots+t_1^2-r^2)\\
&\quad& +r^2\Del_{t_k}+t_kr\Del_r-t_k \quad (1 \leq k \leq n)
\end{eqnarray*}
Then, the characteristic ideal ${\rm in}_{(0, e)}(I)$ contains the polynomials
\begin{eqnarray*}
&&     \xi_{x_{ij}}                    \, (1\leq i \leq j \leq n+1), 
\quad  \xi_{y_i}                       \, (1\leq i \leq n+1), 
\quad  t_{n+1}^2+\cdots+t_1^2-r^2, \\
&&     t_i\xi_{t_j}-t_j\xi_{t_i}       \, (1\leq i < j \leq n+1), 
\quad  r^2\xi_{t_i}+t_ir\xi_r(1 \leq i \leq n+1).
\end{eqnarray*}
Let $I'$ be the ideal in the polynomial ring
$
\Comp [
x, y, r, t, \xi_x, \xi_y, \xi_r, \xi_t
]
$
generated by these polynomials.
Then, we have $I'\subset {\rm in}_{(0,e)}(I)$.
Since $\dim I' \geq \dim {\rm in}_{(0,e)}(I)$,
it is enough to show that $\dim I' \leq N$.

Consider the graded reverse lexicographic order 
satisfying
$$
\xi_{t_{n+1}} \succ \cdots \succ \xi_{t_1} 
\succ \xi_x \succ \xi_y \succ \xi_r \succ
t_{n+1} \succ \cdots \succ t_1 
\succ x \succ     y \succ     r.
$$
Since the degree of the Hilbert polynomial of an ideal in the polynomial ring
equals 
that of the initial ideal with respect to the graded order of the ideal
(see, e.g., \cite[p448, Proposition 4]{CLO}), 
the dimension of $I'$ is equal to that of the initial ideal $\LT_\prec(I')$
with respect to this order.
The initial ideal 
$
\LT_\prec(I')
$
contains the monomials
$
\xi_{x_{ij}}, \xi_{y_i}, 
t_i\xi_{t_j}, 
r^2\xi_{t_i}, 
t_{n+1}^2.
$
Let $I''$ be the ideal generated by these monomials.
Analogously, we can show that it suffices to prove that 
the dimension of $I''$ is not more than $N$.

Computing the irreducible decomposition of the algebraic variety defined by
$I''$ as 
\begin{eqnarray*}
&&V(\xi_{x_{kl}}, \xi_{y_k}, t_i\xi_{t_j}, r^2\xi_{t_k}, t_{n+1}^2;
  1\leq k\leq l\leq n+1, 1\leq i<j\leq n+1
  )\\
&=&
V(\xi_{x_{ij}}, \xi_{y_i}, t_{n+1};1\leq i\leq j\leq n+1)
  \cap \bigcap_{1\leq i< j\leq n+1}V(t_i\xi_{t_j})
  \cap \bigcap_{i=1}^{n+1}V(r^2\xi_{t_i})\\
&=&
  V(\xi_{x_{ij}}, \xi_{y_i}, t_{n+1};1\leq i\leq j\leq n+1)
  \cap
  \bigcup_{i=1}^{n+1}V(t_1, \dots, t_{i-1}, \xi_{t_{i+1}}, \dots, \xi_{t_{n+1}})\\
  &\quad&\cap 
  \left(
    V(r) \cup V(\xi_{t_1}, \dots, \xi_{t_{n+1}})
  \right)\\
&=&
  \left(\bigcup_{k=1}^{n+1}
    V(\xi_{x_{ij}}, \xi_{y_i}, t_{n+1}, 
	  t_1, \dots, t_{k-1}, \xi_{t_{k+1}}, \dots, \xi_{t_{n+1}}
    )
  \right)\\
  &\quad&\cap 
  \left(
    V(r) \cup V(\xi_{t_1}, \dots, \xi_{t_{n+1}})
  \right)\\
&=&
  \left(\bigcup_{i=1}^{n+1}
    V(\xi_{x_{kl}}, \xi_{y_l}, r, t_{n+1}, 
	  t_1, \dots, t_{i-1}, \xi_{t_{i+1}}, \dots, \xi_{t_{n+1}};
      1\leq k\leq l\leq n+1
    )
  \right)\\
  &\quad&\cup
  \left(\bigcup_{i=1}^{n+1}
    V(\xi_{x_{kl}}, \xi_{y_k}, t_{n+1}, 
	  t_1, \dots, t_{i-1}, \xi_{t_1}, \dots, \xi_{t_{n+1}};
      1\leq k\leq l\leq n+1
    )
  \right),
\end{eqnarray*}
we conclude that the dimension of $I''$ is exactly $N$.
\qedhere
\end{proof}
\section{Holonomic ideal annihilating $\exp(g)\mu_r$}\label{section3}
Let $g(x, y, t)$ be the polynomial
$
\sum_{1\leq i \leq j \leq n+1}x_{ij}t_it_j 
+\sum_{i=1}^{n+1}y_it_i.
$
We can get a holonomic ideal annihilating 
the distribution $\exp(g(x, y, t))\mu_r$
by the following lemma.
\begin{lem}
Consider the ring of differential operators with polynomial coefficients 
$\Comp\langle x_1, \dots, x_n, \Del_1, \dots, \Del_n \rangle$.
Let $u$ be a distribution and 
suppose that $I\subset \Ann(u)$ is a holonomic ideal. 
Let $f$ be a polynomial and $f_i :=\Del f/\Del x_i$.
Then, the left ideal $J$ generated by 
$$
\left\{
P(x_1, \dots, x_n;\Del_{x_1}-f_1, \dots, \Del_{x_n}-f_n)
\vert
P(x_1, \dots, x_n;\Del_{x_1}, \dots, \Del_{x_n})\in I
\right\}
$$
is a holonomic ideal such that $J\subset \Ann(e^fu)$
\end{lem}
For a proof of this lemma, we refer to \cite{OST}.
It follows from this lemma that 
the left ideal $J$ in $D$ generated by the following differential operators 
is a holonomic ideal and included in $\Ann(\exp(g)\mu_r)$.
\begin{equation}\label{genJ0}
  \begin{split}
	&\Del_{x_{ij}} -t_it_j  \quad (1\leq i \leq j \leq n+1), \\
	&\Del_{y_i} - t_i       \quad (1\leq i \leq n+1), \\
	&t_i(\Del_{t_j}-\sum_{k=1}^{n+1}x_{jk}t_k-x_{jj}t_j-y_j)
	-t_j(\Del_{t_i}-\sum_{k=1}^{n+1}x_{ik}t_k-x_{ii}t_i-y_i)\\
    & \quad (1\leq i < j \leq n+1), \\
	&t_1^2+\cdots +t_{n+1}^2-r^2, \\
	&r\Del_r
		+1
		+\sum_{i=1}^{n+1}t_i
		\left(
        \Del_{t_i}-\sum_{k=1}^{n+1}x_{ik}t_k-x_{ii}t_i-y_i
		\right)
  \end{split}
\end{equation}
In fact, we will show that 
the ideal $J$ is generated by the differential operators
\begin{equation}\label{GenJ}
\begin{split}
  &t_i - \Del_{y_i}                  \, (1\leq i \leq n+1), \quad
  \Del_{x_{ij}}-\Del_{y_i}\Del_{y_j}  \, (1\leq i \leq j \leq n+1), \\
  &\sum_{i=1}^{n+1}\Del_{x_{ii}}-r^2,                                     \\
  &x_{ij}\Del_{x_{ii}}+2(x_{jj}-x_{ii})\Del _{x_{ij}}
    -x_{ij}\Del_{x_{jj}}                                                   \\
	&\quad	+\sum _{k \neq i, j}
	\left( x_{kj}\Del_{x_{ik}}-x_{ik}\Del_{x_{jk}}\right)
	+y_j\Del_{y_i} -y_i\Del _{y_j}
	+\Del_{t_i}\Del_{y_j}-\Del_{t_j}\Del_{y_i}                             \\
	&\quad (1\leq i < j \leq n+1, x_{k\ell}=x_{\ell k}),                  \\ 
  &r\Del_r -2\sum _{i\leq j}x_{ij}\Del_{x_{ij}}
    -\sum_{i=1}^{n+1} y_i\Del _{y_i} -n 
    +\sum_{i=1}^{n+1} \Del_{t_i}\Del_{y_i}
\end{split}
\end{equation}
To prove this statement, we prepare the following lemma.
\begin{lem}
\begin{equation}\label{mod1}
t^\alpha \equiv \Del_y^\alpha  
\quad \bmod \, D\{t_i-\Del_{y_i};1\leq i \leq n+1\}
\end{equation}
\end{lem}
\begin{proof}
When $\alpha = e_i$, the equation $(\ref{mod1})$ obviously holds.
Let us assume that $(\ref{mod1})$ holds for the case of $\alpha-e_i$. 
Then, we have
\begin{eqnarray*}
t^\alpha 
&=& t_it^{(\alpha-e_i)}\\
&\equiv& t_i\Del_y^{(\alpha-e_i)}
         \quad \bmod \, D\{t_i-\Del_{y_i};1\leq i \leq n+1\}\\
&=& \Del_y^{(\alpha-e_i)}t_i
 =  \Del_y^{(\alpha-e_i)}(t_i-\Del_{y_i})
         +\Del_y^{(\alpha-e_i)}\Del_{y_i}\\
&\equiv& \Del_y^{(\alpha-e_i)}\Del_{y_i}
		 \quad \bmod \, D\{t_i-\Del_{y_i};1\leq i \leq n+1\}\\
&=& \Del_y^\alpha
\end{eqnarray*}
Hence, $(\ref{mod1})$ holds for $\alpha$.
Therefore, the equation $(\ref{mod1})$ holds for any $\alpha$.
\end{proof}

Finally, we prove the following lemma.
\begin{lem}
The differential operators $(\ref{GenJ})$ generates $J$.
\end{lem}
\begin{proof}
Let $K$ be the left ideal generated by $(\ref{GenJ})$.
First, let us show $J\subset K$.
The equation  
\begin{equation}\label{mod2}
\Del_{x_{ij}}-t_it_j \equiv \Del_{x_{ij}}-\Del_{y_i}\Del_{y_j}
		 \quad \bmod \, D\{t_i-\Del_{y_i};1\leq i \leq n+1\}
\end{equation}
gives the inclusion
$
\Del_{x_{ij}}-t_it_j \in K.
$

The inclusion
$
\Del_{y_i} - t_i \in K
$
is obvious.
The inclusion
$
t_i(\Del_{t_j}-\sum_{k=1}^{n+1}x_{jk}t_k-x_{jj}t_j-y_j)
-t_j(\Del_{t_i}-\sum_{k=1}^{n+1}x_{ik}t_k-x_{ii}t_i-y_i)
\in K
$
follows from 
\begin{eqnarray*}
  & &t_i(\Del_{t_j}-\sum_{k=1}^{n+1}x_{jk}t_k-x_{jj}t_j-y_j)
  -t_j(\Del_{t_i}-\sum_{k=1}^{n+1}x_{ik}t_k-x_{ii}t_i-y_i)\\
  &=&
  \sum_{k=1}^{n+1}x_{ik}t_kt_j-\sum_{k=1}^{n+1}x_{jk}t_kt_i
  -x_{jj}t_it_j+x_{ii}t_it_j
  -y_jt_i+y_it_j
  +t_i\Del_{t_j}-t_j\Del_{t_i}\\
  &=&
  \sum_{k=1}^{n+1}(x_{ik}t_j-x_{jk}t_i)t_k
  +(x_{ii}-x_{jj})t_it_j
  +y_it_j-y_jt_i
  +t_i\Del_{t_j}-t_j\Del_{t_i}  \\
  &\equiv&
  \sum_{k=1}^{n+1}(x_{ik}\Del_{y_j}-x_{jk}\Del_{y_i})\Del_{y_k}  
  +(x_{ii}-x_{jj})\Del_{y_i}\Del_{y_j}\\
  &\quad& +y_i\Del_{y_j}-y_j\Del_{y_i}
  +\Del_{y_i}\Del_{t_j}-\Del_{y_j}\Del_{t_i}
  \quad \bmod \, D\{t_i-\Del_{y_i};1\leq i \leq n+1\}\\
  &\equiv&
  \sum_{k=1}^{n+1}(x_{ik}\Del_{x_{jk}}-x_{jk}\Del_{x_{ik}})
  +(x_{ii}-x_{jj})\Del_{x_{ij}}\\
  &\quad& +y_i\Del_{y_j}-y_j\Del_{y_i}
  +\Del_{y_i}\Del_{t_j}-\Del_{y_j}\Del_{t_i}
  \quad \bmod\, D\{\Del_{x_{ij}}-\Del_{y_i}\Del_{y_j};1\leq i \leq j\leq n+1\}\\
  &=&
  x_{ij}\Del_{x_{jj}}
  +\sum_{k\neq i, j}(x_{ik}\Del_{x_{jk}}-x_{jk}\Del_{x_{ik}})
  -x_{ij}\Del_{x_{ii}}
  +2(x_{ii}-x_{jj})\Del_{x_{ij}}\\
  &\quad& +y_i\Del_{y_j}-y_j\Del_{y_i}
  +\Del_{y_i}\Del_{t_j}-\Del_{y_j}\Del_{t_i}.
\end{eqnarray*}
Since
\begin{eqnarray*}
&&
  t_1^2+\cdots +t_{n+1}^2-r^2\\
&\equiv&
  \Del_{y_1}^2+\cdots +\Del{y_{n+1}}^2-r^2
  \quad \bmod \, D\{t_i-\Del_{y_i};1\leq i \leq n+1\}  \\
&=&
  \sum_{i=1}^{n+1}\Del_{x_{ii}}-r^2
  \quad \bmod\, D\{\Del_{x_{ij}}-\Del_{y_i}\Del_{y_j};1\leq i \leq j\leq n+1\}, \\
\end{eqnarray*}
we have
$
\sum_{i=1}^{n+1}\Del_{x_{ii}}-r^2 \in K.
$

The inclusion
$
  r\Del_r
  +1
  +\sum_{i=1}^{n+1}t_i
     \left(
          \Del_{t_i}-\sum_{k=1}^{n+1}x_{ik}t_k-x_{ii}t_i-y_i
     \right)
\in K
$
follows from 
\begin{eqnarray*}
&&
  r\Del_r
  +1
  +\sum_{i=1}^{n+1}t_i
     \left(
          \Del_{t_i}-\sum_{k=1}^{n+1}x_{ik}t_k-x_{ii}t_i-y_i
     \right)\\
&=&
	 r\Del_r
	 +1
	 +\sum_{i=1}^{n+1}t_i\Del_{t_i}
	 -\sum_{i=1}^{n+1}\sum_{k=1}^{n+1}x_{ik}t_it_k
	 -\sum_{i=1}^{n+1}x_{ii}t_i^2
     -\sum_{i=1}^{n+1}y_it_i\\
&=&
	 r\Del_r
	 -n
	 +\sum_{i=1}^{n+1}\Del_{t_i}t_i
	 -\sum_{i=1}^{n+1}\sum_{k=1}^{n+1}x_{ik}t_it_k
	 -\sum_{i=1}^{n+1}x_{ii}t_i^2
     -\sum_{i=1}^{n+1}y_it_i\\
&\equiv&
	 r\Del_r
	 -n
	 +\sum_{i=1}^{n+1}\Del_{t_i}\Del_{y_i}
	 -\sum_{i=1}^{n+1}\sum_{k=1}^{n+1}x_{ik}\Del_{x_{ik}}
	 -\sum_{i=1}^{n+1}x_{ii}\Del_{x_{ii}}
     -\sum_{i=1}^{n+1}y_i\Del_{y_i}\\
	 &&  \quad \bmod\, 
	 D\{t_i-\Del_{y_i}, \Del_{x_{ij}}-\Del_{y_i}\Del_{y_j};
	 1\leq i \leq j\leq n+1\}\\
&=&
	 r\Del_r
	 -n
	 +\sum_{i=1}^{n+1}\Del_{t_i}\Del_{y_i}
	 -2\sum_{i\leq j}x_{ij}\Del_{x_{ij}}
     -\sum_{i=1}^{n+1}y_i\Del_{y_i}.
\end{eqnarray*}
Therefore, we have $J\subset K$.

Next, let us show the opposite inclusion $K\subset J$.
The inclusion
$
t_i - \Del_{y_i} \in J
$
is obvious.
The inclusion
$
\Del_{x_{ij}}-\Del_{y_i}\Del_{y_j} \in J
$
follows from equation $(\ref{mod2})$.
Other generators of $K$ are also in $J$ because of the above equivalence
relation.
\qedhere
\end{proof}
\section{The Fisher--Bingham Integral}\label{section4}
Let $D'$ be the ring of differential operators with polynomial coefficients 
$\Comp\langle x, y, r, \Del_x, \Del_y, \Del_r\rangle$.
The left ideal 
$
J'
:=
D'\cap(J+\{\Del_{t_1}, \dots, \Del_{t_{n+1}}\}\cdot D)
$ in $D'$
is the integration ideal of $J$.
The Fisher--Bingham integral $(\ref{FB})$ can be written as
$$
F(x, y, r) 
=
\langle e^{g(x, y, t)}\mu_r, 1\rangle
=
\int_{\Real^{n+1}} \exp(g(x, y, t))\mu_r dt.
$$
Hence, the operators in $J'$ annihilate $F(x, y, r)$.
It is known that the integration ideal of a holonomic ideal is also
a holonomic ideal (see, e.g.,  \cite[2, chapter1]{Bjork}).
Therefore, if we obtain a set of generators of $J'$, then this set generates
a holonomic ideal.
In this section, we compute a set of generators of $J'$.
As the first step, we prove the following lemma.
\begin{lem}
Let $P$ be an arbitrary differential operator in $(\ref{GenJ})$; then 
we have
$$
t^\alpha P \equiv \Del_y^\alpha P 
\quad \bmod \, D\{t_i-\Del_{y_i};1\leq i \leq n+1\}.
$$
\end{lem}
\begin{proof}
  For simplicity, put $Q_{ij} =
  x_{ij}\Del_{x_{ii}}+2(x_{jj}-x_{ii})\Del _{x_{ij}}
  -x_{ij}\Del_{x_{jj}}
  +\sum _{k \neq i, j}
  \left( x_{kj}\Del_{x_{ik}}-x_{ik}\Del_{x_{jk}}\right)
  $
  and 
  $
  R=
  r\Del_r -2\sum _{i\leq j}x_{ij}\Del_{x_{ij}}-n
  $.
  The following equations prove the lemma.
  \begin{eqnarray*}
	&&
	t^\alpha \left(
	Q_{ij}
	+y_j\Del_{y_i} -y_i\Del _{y_j}
	+\Del_{t_i}\Del_{y_j}-\Del_{t_j}\Del_{y_i}\right)\\
	&=&
	\left(
	Q_{ij}
	+y_j\Del_{y_i} -y_i\Del _{y_j}
	+\Del_{t_i}\Del_{y_j}-\Del_{t_j}\Del_{y_i}\right)t^\alpha
	-\alpha_i\Del_{y_j}\Del_t^{(\alpha-e_i)}
	+\alpha_j\Del_{y_i}\Del_t^{(\alpha-e_j)}\\	
	&\equiv&
	\left(
	Q_{ij}
	+y_j\Del_{y_i} -y_i\Del _{y_j}
	+\Del_{t_i}\Del_{y_j}-\Del_{t_j}\Del_{y_i}
	\right)\Del_y^\alpha\\
	&\quad&
	-\alpha_i\Del_{y_j}y^{(\alpha-e_i)}
	+\alpha_j\Del_{y_i}y^{(\alpha-e_j)}
	\quad \bmod \, D\{t_i-\Del_{y_i};1\leq i \leq n+1\}	\\	
	&=&
	\Del_y^\alpha\left(
	Q_{ij}
	+y_j\Del_{y_i} -y_i\Del _{y_j}
	+\Del_{t_i}\Del_{y_j}-\Del_{t_j}\Del_{y_i}\right),\\
&&
  t^\alpha\left(
  R
  -\sum_{i=1}^{n+1} y_i\Del _{y_i}
  +\sum_{i=1}^{n+1} \Del_{t_i}\Del_{y_i}
  \right)\\
&=&
  \left(
  R
  -\sum_{i=1}^{n+1} y_i\Del _{y_i}
  +\sum_{i=1}^{n+1} \Del_{t_i}\Del_{y_i}
  \right)t^\alpha
  -\sum_{i=1}^{n+1} \alpha_i\Del_{y_i}t^{(\alpha -e_i)}\\
&\equiv&
  \left(
  R
  -\sum_{i=1}^{n+1} y_i\Del _{y_i}
  +\sum_{i=1}^{n+1} \Del_{t_i}\Del_{y_i}
  \right)\Del_y^\alpha
  -\sum_{i=1}^{n+1} \alpha_i\Del_{y_i}\Del_y^{(\alpha -e_i)}
  \quad \bmod \, D\{t_i-\Del_{y_i};1\leq i \leq n+1\}\\
&=&
  \Del_y^\alpha\left(
  R
  -\sum_{i=1}^{n+1} y_i\Del _{y_i}
  +\sum_{i=1}^{n+1} \Del_{t_i}\Del_{y_i}
  \right).
\end{eqnarray*}
\qedhere
\end{proof}
\begin{thm}
The integration ideal
$J'$ is generated by the differential operators in $(\ref{annFB}).$
\end{thm}
\begin{proof}
Let $F$ and $F'$ be the set consisting of the differential operators 
$(\ref{GenJ})$ and $(\ref{annFB})$ respectively.
The inclusion $D'\cdot F' \subset J'$ is obvious.
We need to show the opposite inclusion
$D'\cdot F' \supset J'$.
If a differential operator $P$ is contained in $J'$, 
then $P$ can be written as
$$
P = \sum_i Q_iP_i + \sum_j \Del_{t_j}R_j \quad
(P_i\in F, \, Q_i \in D, R_j\in D), 
$$
from the definition of $J'$.
Without loss of generality, we can assume that no term of $Q_i$ contains
$\Del_t$.
Note that
$$
t^\alpha P_i \equiv \Del_y^\alpha P_i 
\quad \bmod \, D\{t_k-\Del_{y_k};1\leq k \leq n+1\}; 
$$
then, $P$ can be written as 
$$
P = \sum_i Q'_iP_i + \sum_j \Del_{t_j}R_j + \sum_k S_k(t_k-\Del_{y_k}) \quad
(P_i\in F, \, Q_i'\in D', \, R_j \in D, S_k \in D). 
$$
Since all differential operators in $F$ except $t_i-\Del_{y_i}$ have the form
$
P' +\sum_i \Del_{t_i}U'_i \quad (P' \in F', \, U'_i \in D'), 
$
$P$ can be written as
$$
P = \sum_i Q'_iP'_i + \sum_j \Del_{t_j}R_j + \sum_k S_k(t_k-\Del_{y_k}) \quad
(P_i\in F, \, Q_i'\in D', \, R_j \in D, S_k \in D).
$$
Moving some terms to the left-hand side, we obtain
$$
P -\sum_i Q_i'P_i' - \sum_k S_k(t_k-\Del_{y_k})= \sum_j \Del_{t_j}R_j  \quad
(P_i'\in F', \, Q_i'\in D', \, R_j \in D, S_k \in D) 
$$
Without loss of generality, if we assume that no term of $S_k$ contains
$\Del_t$,
then the left-hand side of the equation does not contain $\Del_t$.
Expanding both sides and comparing the coefficients, we get 
$\sum_j \Del_{t_j}R_j=0$, in other words, we obtain
$$
P -\sum_i Q_i'P_i' = \sum_k S_k(t_k-\Del_{y_k})\quad
(P_i'\in F', \, Q_i'\in D', \,S_k \in D ). 
$$
The right-hand side of this equation is included in the left ideal 
$D\cdot\{t_i-\Del_{y_i}\vert 1\leq i \leq n+1\}$
in $D$.
Let the weight of $t_i$ be $1$ and that of other variables be $0$, 
and consider a term order $\prec$ with this weight.
The Gr\"{o}bner basis of 
$D\cdot\{t_i-\Del_{y_i}\vert 1\leq i \leq n+1\}$
with this order is
$\{t_i-\Del_{y_i}\vert 1\leq i \leq n+1\}, $
and the initial ideal is generated by
$\{t_i\vert 1\leq i \leq n+1\}.$
Hence, the leading term of $P -\sum Q_k'P_k' \in D'$ 
with respect to the order $\prec$
must divide some $t_i$.
However, the differential operator in $D'$ 
which satisfies this condition is only $0$.
Then, we have $P\in D'F'$.
\qedhere
\end{proof}
\begin{cor}
The integration ideal $J'$ is a holonomic ideal.
\end{cor}

Affiliation: Department of Mathematics, Kobe University and JST crest Hibi projict\\
E--mail address: tkoyama@math.kobe-u.ac.jp
\end{document}